\documentclass[12pt, a4paper]{article}

\usepackage{amsmath}
\usepackage{amsthm}
\usepackage{amssymb}
\usepackage{color}

\newtheorem{theorem}{Theorem}

\newtheorem{definition}[theorem]{Definition}

\newtheorem{corollary}[theorem]{Corollary}

\newtheorem{assumption}[theorem]{Assumption}

\title{A note on large deviations in life insurance}
\author{Stefan Gerhold\thanks{Financial support from the Austrian Science Fund (FWF) under grant P~30750 is gratefully acknowledged.}\\
TU Wien \\
1040 Vienna, Austria\\
\tt{sgerhold@fam.tuwien.ac.at} 
	}

\date{\today}

\numberwithin{equation}{section}
\numberwithin{theorem}{section}

\begin{document}

\maketitle

\begin{abstract}
We study large and moderate deviations for a life insurance portfolio, without assuming identically distributed losses. The crucial assumption is that losses are bounded, and that variances are bounded below. From a standard large deviations upper bound, we get an exponential bound for the probability of the average loss exceeding a threshold. A counterexample shows that a full large deviation principle does not follow from our assumptions.
\end{abstract}

MSC classes: 60F10, 91B30

\section{Introduction and main assumption}

Let $L_k$ be the loss of  the $k$th contract of a life insurance portfolio,
that is, a random variable
which aggregates the discounted remaining cash flows. With $X_k=L_k - \mathbb{E}[L_k],$ define
\begin{equation}\label{eq:Mn}
  M_n = \frac1n \sum_{k=1}^n X_k,
\end{equation}
the average centered loss of the portfolio.
 We are interested in estimates for
\begin{equation}\label{eq:prob}
  \mathbb{P}[ M_n \geq x], \quad x>0\ \text{fixed,}\ n\to\infty,
\end{equation}
the probability of an event that is typically calamitous for the insurer.
Somewhat surprisingly, conditions on $(X_k)_{k\in\mathbb N}$ that yield such estimates
have apparently not been made explicit in the literature. Note that the well-known
Cram\'er's theorem does not settle this problem, because the assumption
of identically distributed losses is not suitable for life insurance.
Also, it is not immediately clear under which assumptions on the~$X_k$
the G\"artner-Ellis theorem can be applied. Still, we will see that it is not very hard to obtain
bounds and estimates for~\eqref{eq:prob} from known large and moderate deviation
results, and so much of this note has a didactic character,
except possibly Theorem~\ref{thm:counter}.
When justifying the assumptions we make below, we focus on life insurance, but our
observations may also apply to other risk aggregation problems.
Recall the following standard definition:
\begin{definition}
  A sequence of random variables~$(Z_n)_{n\in\mathbb N}$ satisfies the LDP (large deviation principle)
  with good rate function~$I$ and speed~$s(n)$, if
  \begin{itemize}
    \item[(i)] $I:\mathbb R \to [0,\infty]$ is not infinite everywhere,
    and the level sets $\{x: I(x)\leq c\},$ $c\in[0,\infty),$ are compact. In particular, $I$ is lower semi-continuous.
    \item[(ii)]  $s(n)>0$ satisfies $\lim_{n\to\infty}s(n)=\infty.$
    \item[(iii)] For any Borel set~$G,$
    \begin{align*}
      -I(\mathrm{int}(G)) &\leq \liminf_{n\to\infty} \frac{1}{s(n)}
        \log \mathbb{}P[Z_n \in G] \\
        &\leq
        \limsup_{n\to\infty} \frac{1}{s(n)}
        \log \mathbb{}P[Z_n \in G] \leq -I(\mathrm{cl}(G)),
    \end{align*}
    where $I(A):=\inf_{x\in A}I(x)$ for any $A\subseteq \mathbb R.$
  \end{itemize}
\end{definition}
One of the central results in LD theory is Cram\'er's theorem (Theorem~2.2.3
in~\cite{DeZe98}), which asserts
that the sequence of empirical means $M_n$ satisfies an LDP under the assumption that the $X_k$ are iid.
However, in life insurance the loss distributions depend significantly on several parameters including
amount insured, age, and time to expiry, which contradicts the assumption
of identical distributions. We now state a different set of conditions, argue
why they seem reasonable, and subsequently
explore the estimates they imply.
\begin{assumption}\label{ass:main}
  \begin{itemize}
    \item[(i)] $(X_k)_{k\in\mathbb N}$ is a sequence of independent centered real random variables,
    \item[(ii)] there is $c_0>0$ such that $|X_k|\leq c_0$ for all~$k$,
    \item[(iii)] there is $c_1>0$ such that $\mathrm{Var}[X_k]=\mathbb{E}[X_k^2]\geq c_1$ for all~$k$.
  \end{itemize}
\end{assumption}
In the life insurance application described above,
the independence assumption ignores certain risks, such as epidemics and natural disasters,
but still seems reasonable for large portfolios. Part~(ii) is usually satisfied in practice, as insurers
prescribe an upper limit on the possible amount insured. As for~(iii), note
that clearly we may assume $\mathrm{Var}[X_k]>0,$ because it makes no sense to include
contracts with no remaining random cash flows. Then, since there is usually a lowest possible
amount insured, and there are only finitely many value combinations for the parameters age, time to expiry,
sex, and type of insurance, a uniform lower bound on the loss variance is natural.
Of course, for continuous-time models, which are not widespread
in practice anyways, this applies only after time discretization.

There is a large literature on large deviations for compound sums and
more sophisticated models in risk theory, but apparently not for the individual
risk model with non-identical distributions. In practice, premia and reserves are calculated
for each contract separately, i.e.\ using an individual model.
For computing the distribution of~\eqref{eq:Mn}
numerically, e.g.\ to compute value at risk,
 a standard approach is to pass to a collective model. For an asymptotic
approximation of~\eqref{eq:prob}, which is our goal, such a change of model is not required.
Large deviations for an individual model of credit and insurance risk are also
studied in~\cite{DeDeDu04}, but their assumptions are different from ours.

The rest of this note is structured as follows.
From a practical viewpoint, our
main result is Theorem~\ref{thm:bd}, which shows that Assumption~\ref{ass:main}
yields an exponentially small upper bound, which is weaker than a full LDP,
but should suffice for practical purposes. 
In Theorem~\ref{thm:counter}, we show that Assumption~\ref{ass:main}
does not suffice to establish an LDP for~$M_n$. Theorem~\ref{thm:ge} adds a somewhat
restrictive assumption, which implies an LDP. Finally, Corollary~\ref{cor:md} establishes
 moderate deviation estimates for~$M_n.$

\section{Large deviations: an upper bound}

For practical purposes, an upper bound for~\eqref{eq:prob}
is much more important than a lower bound.
 We now show
that the -- rather weak -- Assumption~\ref{ass:main} implies
an exponential upper estimate.

\begin{theorem}\label{thm:bd}
   Under Assumption~\ref{ass:main}, there exists a positive function
   $J:(0,\infty) \to (0,\infty)$ such that
   \begin{equation}\label{eq:ld upper}
     \limsup_{n\to\infty}\frac1n \log
     \mathbb{P}[M_n\geq x] \leq
     -J(x), \quad x>0.
   \end{equation}
\end{theorem}
\begin{proof}
   We apply the general LD upper bound from Theorem~4.5.20 in~\cite{DeZe98}.
   Define
   \[
      \bar{\Lambda}(\lambda):=
       \limsup_{n\to\infty}\frac1n \log \mathbb{E}[e^{\lambda n M_n}]
       =\limsup_{n\to\infty}\frac1n \sum_{k=1}^n \log \mathbb{E}[e^{\lambda X_k}].
   \]
    Since
   $|M_n|\leq c_0$ is bounded, the sequence of its laws is exponentially
   tight (definition on p.~8 of~\cite{DeZe98}).
  Thus, part~(a) of
   Theorem~4.5.20 in~\cite{DeZe98} implies
   \[
       \limsup_{n\to\infty}\frac1n \log
     \mathbb{P}[M_n\geq x] \leq - \inf_{y\geq x}\bar{\Lambda}^*(y)
     =: -J(x),
   \]
   where
     \[
    \bar{\Lambda}^*(x) := \sup_{\lambda\in \mathbb R}\big(\lambda x-\bar{\Lambda}(\lambda)\big), \quad x\in\mathbb R,
  \]   
    is the Fenchel-Legendre transform  of~$\bar{\Lambda}$.
   The key point now is to show that~$J$ is positive, because
  otherwise~\eqref{eq:ld upper} would be of little use. Since $\bar{\Lambda}^*$
  is convex (see Theorem~4.5.3~(a) in~\cite{DeZe98}), it suffices
  to show that $\bar{\Lambda}^*$ is positive on some interval $(0,\delta)$
  with $\delta>0.$
   By Assumption~\ref{ass:main},
   \[
      \mathbb{E}[e^{\lambda X_k}] =
      1+\tfrac12 \mathbb{E}[X_k^2] \lambda^2 
    + \mathrm{O}(\lambda^3), \quad \lambda\to 0,
   \]
   where the error term is uniform w.r.t.~$k$. Hence,
   \begin{align*}
      \log \mathbb{E}[e^{\lambda X_k}]&=\tfrac12 \mathbb{E}[X_k^2] \lambda^2 
    + \mathrm{O}(\lambda^3) \\
    &\leq \tfrac12 c_0^2 \lambda^2 + \mathrm{O}(\lambda^3),
   \end{align*}
    and thus
   $\bar{\Lambda}(\lambda) \leq  c_0^2 \lambda^2$ for small~$\lambda.$
   Define the convex function
   \[
     \Theta(\lambda):= c_0\lambda^2  \vee \bar{\Lambda}(\lambda),\quad \lambda\in\mathbb R.
   \]
   Its Fenchel-Legendre transform~$\Theta^*$
   satisfies $\Theta^*(0)=0,$ is strictly convex in a neighborhood of zero,
   and $\Theta^*\leq \bar{\Lambda}^*.$
\end{proof}

\section{Large deviation principle}

We first give a counterexample (in Theorem~\ref{thm:counter}) that shows that Assumption~\ref{ass:main}
does not imply an LDP for the empirical means. In particular, this shows that
the G\"artner-Ellis theorem is not applicable here without additional assumptions,
such as Assumption~\ref{ass:ge} below.

Let $K_1\subset \mathbb{N}$ be a set of natural numbers with lower density~$0$
and upper density~$1,$ i.e.,
\[
  \nu_1(n) := \sharp \{ 1\leq k\leq n : k \in K_1\}
\]
satisfies
\[
  \liminf_{n\to\infty} \frac{\nu_1(n)}{n} = 0 \quad \text{and} \quad
  \limsup_{n\to\infty} \frac{\nu_1(n)}{n} = 1.
\]
For the existence of such a set, see e.g.\ Theorem~3
in~\cite{StTo98}.
Define $K_2 := \mathbb{N}\setminus K_1$ and $\nu_2(n) := n - \nu_1(n).$

\begin{theorem}\label{thm:counter}
   Let $X^{(1)}$ be a random variable that takes the values $-1,1$ with probability~$\tfrac12$
   each, and $X^{(2)}$ analogously with values $-2,2$. Let~$(X_k)_{k\in\mathbb N}$ be a sequence
   of independent random variables  satisfying
   \[
     X_k \stackrel{\mathrm{d}}{=} X^{(i)}, \quad k\in K_i,\ i=1,2.
   \]
   This sequence satisfies Assumption~\ref{ass:main},
   and the sequence of empirical means $M_n=\frac1n \sum_{k=1}^n X_k$
   does not satisfy an LDP.
\end{theorem}
We defer the proof of this theorem to Appendix~\ref{se:proof}. Since the moment generating function of~$X^{(2)}$ dominates
that of~$X^{(1)}$, the upper estimate in
\begin{align}
  -I^{(1)}(x) &\leq \liminf_{n\to\infty}\frac1n \log \mathbb{P}[M_n>x] \label{eq:mix lower}\\
  &\leq \limsup_{n\to\infty}\frac1n \log \mathbb{P}[M_n>x]
  \leq -I^{(2)}(x), \quad x>0, \notag
\end{align}
can be proved by the general upper LD bound we used in the proof of Theorem~\ref{thm:bd}.
 By Cram\'er's theorem, the section means
\begin{equation}\label{eq:sec means}
    M_n^{(i)} := \frac{1}{\nu_i(n)}\sum_{\substack{k=1 \\ k \in K_i}}^n X_k,\quad i=1,2,
\end{equation}
satisfy LDPs with rate functions $I^{(1)},I^{(2)},$ explicitly given
in~\eqref{eq:rf} below. 
The lower estimate~\eqref{eq:mix lower} then easily follows from
\[
  \mathbb{P}[M_n>x] \geq \mathbb{P}\big[M_n^{(1)}>x, \ M_n^{(2)}>x\big]
  = \mathbb{P}\big[M_n^{(1)}>x\big]\mathbb{P}\big[M_n^{(2)}>x\big].
\]
Thus, we have exponential lower and upper bounds, but the highly irregular
interlacement of two distributions in Theorem~\ref{thm:counter}
 precludes
a single rate function governing both. When such behavior
is explicitly forbidden, we can actually obtain a full LDP, using the G\"artner-Ellis theorem.

\begin{assumption}\label{ass:ge}
  \begin{itemize}
     \item[(i)] There is a  partition
     \[
       \mathbb N = N_1 \cup \dots \cup N_p
     \]
     such that for all~$1\leq i\leq p$ and $k\in N_i$, the law of
     $X_k\stackrel{\mathrm{d}}{=}X^{(i)}$
     is independent of~$k$. We write~$\varphi_i$ for the corresponding
     moment generating function $\varphi_i(\lambda)=
     \mathbb{E}[\exp(\lambda X^{(i)})]$.
     \item[(ii)] For each~$i$, the limit
     \[
        d_i := \lim_{n\to\infty} \frac1n \sharp \{ 1\leq k\leq n : k\in N_i \}
     \]
     exists.
     \end{itemize}
\end{assumption}

\begin{theorem}\label{thm:ge}
  Under Assumptions~\ref{ass:main} and~\ref{ass:ge}, the sequence of empirical
  means~$(M_n)_{n\in\mathbb N}$ satisfies an LDP with good rate function
  \[
    \Lambda^*(x) := \sup_{\lambda\in \mathbb R}\big(\lambda x-\Lambda(\lambda)\big), \quad x\in\mathbb R,
  \]
  the Fenchel-Legendre transform of
   \begin{equation}\label{eq:Lambda}
       \Lambda(\lambda) := \sum_{i=1}^p d_i \log \varphi_i(\lambda).
     \end{equation}
\end{theorem}
\begin{proof}
  This result is an easy consequence of the G\"artner-Ellis theorem (Theorem~2.3.6
  in~\cite{DeZe98}). Indeed, here the function~$\Lambda$ from Assumption~2.3.2
  in~\cite{DeZe98} equals
  \begin{align}
    \Lambda(\lambda) &= \lim_{n\to\infty}\frac1n \log \mathbb{E}[e^{\lambda n M_n}] 
    =  \lim_{n\to\infty}\frac1n \sum_{k=1}^n \log \mathbb{E}[e^{\lambda X_k}]\label{eq: Lambda 1} \\
    &= \lim_{n\to\infty}\frac1n \sum_{i=1}^p \sum_{\substack{k=1 \\ k\in N_i}}^n \log \varphi_i(\lambda)
    = \sum_{i=1}^p d_i \log \varphi_i(\lambda), \notag
  \end{align}
  which agrees with~\eqref{eq:Lambda}.
  As the~$X_k$ are bounded by Assumption~\ref{ass:main}, the domain
  of~$\Lambda$ is~$\mathbb R.$ By Remark~(c) on p.~45
  of~\cite{DeZe98}, it is thus not necessary to verify the so-called steepness
  of~$\Lambda.$ Since moment generating functions are
  smooth, so is~$\Lambda$. Therefore,
  all assumptions of the G\"artner-Ellis theorem are satisfied.
\end{proof}

\section{Moderate deviations}

When~$x$ in~\eqref{eq:prob} is allowed to depend on~$n$,
and $n^{-1/2} \ll x \ll 1,$ we are in a regime in between of the CLT
and the LD scalings, which is known as moderate deviations regime.
We need the following result from~\cite{Pe54}, which is also
presented in detail as Theorem~1.1 in~\cite{PeRo08}.
\begin{theorem}[Petrov 1954]\label{thm:pet}
   Let~$(X_k)_{k\in\mathbb N}$ be a sequence of independent centered random variables
   such that there are positive numbers $g,G,H$ with
   \begin{equation}\label{eq:pet ineq}
     g\leq \big|\mathbb{E}[e^{h X_k}]\big|\leq G \ \ \text{in the complex circle}\
     |h|<H,\ k\in\mathbb N.
   \end{equation}
   Moreover, suppose that $B_n:=\sum_{k=1}^n \mathbb{E}[X_k^2]$
   satisfies $\liminf B_n/n>0.$ Then, for $1<y=\mathrm{o}(\sqrt{n}),$
   \begin{equation}\label{eq:pet}
     \mathbb{P}\bigg[B_n^{-1/2}\sum_{k=1}^n X_k > y\bigg]
     =\big(1-\Phi(y)\big)
     \exp\bigg(\frac{y^3}{\sqrt{n}} \lambda_n\Big(\frac{y}{\sqrt{n}}\Big)\bigg)
     \big(1+\mathrm{o}(1)\big)
   \end{equation}
   as $n\to\infty,$ where~$\Phi$ is the standard Gaussian cdf, and $\lambda_n$ is 
   a power series which converges uniformly w.r.t.~$n,$ and with coefficients expressible by the cumulants of the~$X_k.$   
\end{theorem}
In~\cite{PeRo08}, it is mentioned that this is a generalization of
Cram\'er's theorem. Indeed, Theorem~1 in~\cite{Cr38} treats
the scaling on the left hand side of~\eqref{eq:pet} (for the iid case), whereas the LD
scaling result
that is nowadays usually called ``Cram\'er's theorem'' is Theorem~6
in~\cite{Cr38}. We now use Petrov's theorem to give a moderate deviations
estimate for~\eqref{eq:prob}.
The first estimate, \eqref{eq:M MD}, directly follows from
Theorem~\ref{thm:pet}, and thus the scaling  involves~$B_n$.
The simpler scaling in~\eqref{eq:M MD bounds} yields
a slightly cruder estimate, in terms of a lower und an upper bound.
If the parameter~$\alpha$ is close to~$\tfrac12,$ the regime
becomes similar to the LD scaling, which would correspond
to $\alpha=\tfrac12$.
\begin{corollary}\label{cor:md}
   Let~$(X_k)_{k\in\mathbb N}$ be a sequence of random variables satisfying
   Assumption~\ref{ass:main}. For $c>0,$ $\alpha\in(0,\tfrac12),$
   and~$B_n=\sum_{k=1}^n \mathbb{E}[X_k^2]$, we have
   \begin{equation}\label{eq:M MD}
     \mathbb{P}\big[M_n> c n^{\alpha-1}B_n^{1/2}\big]
     =\exp\Big({-\tfrac12} c^2 n^{2\alpha}\big(1+\mathrm{o}(1)\big)\Big).
   \end{equation}
   Moreover, with $c_0$ and $c_1$ as in Assumption~\ref{ass:main}, the bounds
  \begin{equation}\label{eq:M MD bounds}
    \mathbb{P}\big[M_n> c c_0 n^{\alpha-1/2}\big]
     \leq\exp\Big({-\tfrac12} c^2 n^{2\alpha}\big(1+\mathrm{o}(1)\big)\Big)
     \leq\mathbb{P}\big[M_n> c c_1^{1/2} n^{\alpha-1/2}\big]
   \end{equation}
   hold.
\end{corollary}
\begin{proof}  
  Condition~\eqref{eq:pet ineq} is satisfied with $H=c_0^{-1},$
  $g=\tfrac12 e^{-c_0H},$ and $G=e^{c_0H}.$ Indeed, the upper bound is
  clear, and the lower bound follows from
  \[
    \big|\mathbb{E}[e^{h X_k}]\big| \geq 
    \mathbb{E}\big[e^{\mathrm{Re}(h) X_k}\cos(\mathrm{Im}(h) X_k)\big]
  \]
  and
  \[
     \cos(\mathrm{Im}(h) X_k) \geq 1-\tfrac12 \big(\mathrm{Im}(h) X_k\big)^2
\geq 1- \tfrac12(c_0H)^2 = \tfrac12.
  \]
 The condition for~$B_n$ follows from part~(iii) of
 Assumption~\ref{ass:main}.
  We can thus apply Theorem~\ref{thm:pet}, with $y=cn^\alpha.$
   The main contribution
  arises from the factor
  \[
     1-\Phi(y) = \exp\Big({-\tfrac12} c^2 n^{2\alpha}\big(1+\mathrm{o}(1)\big)\Big).
  \]
  Since the convergence of~$\lambda_n$ is uniform, we have
  $\lambda_n(y/\sqrt{n})=\mathrm{O}(1),$ and thus
  \[
      \frac{y^3}{\sqrt{n}} \lambda_n\Big(\frac{y}{\sqrt{n}}\Big)
      = \mathrm{O}(n^{3\alpha-1/2}) \ll n^{2\alpha}.
  \]
  This proves~\eqref{eq:M MD}. For the second assertion, it then suffices
  to note that Assumption~\ref{ass:main} implies
  \[
     c_1 n \leq B_n \leq c_0^2n, \quad n\in\mathbb N. \qedhere
  \]
\end{proof}
Of course, Theorem~\ref{thm:pet} yields further lower order terms in~\eqref{eq:M MD}, if desired.

\appendix

\section{Proof of Theorem~\ref{thm:counter}}\label{se:proof}

   It is obvious that Assumption~\ref{ass:main} is satisfied.
   By Cram\'er's theorem, $M_n^{(1)},M_n^{(2)},$ defined in~\eqref{eq:sec means},
    satisfy LDPs with good rate functions
   \begin{align}\label{eq:rf}
     \begin{split}
     I^{(1)}(x) &=
     \begin{cases}
       \log 2 + \tfrac{x+1}{2} \log \tfrac{x+1}{2}+ \tfrac{1-x}{2}\log \tfrac{1-x}{2}, 
       & x\in[-1,1], \\
       \infty & \text{otherwise},
     \end{cases} \\
   I^{(2)}(x) &=
     \begin{cases}
       \log 2 + \tfrac{x+2}{4} \log \tfrac{x+2}{4}+ \tfrac{2-x}{4}\log \tfrac{2-x}{4}, 
       & x\in[-2,2], \\
       \infty & \text{otherwise},
     \end{cases}
     \end{split}
   \end{align}
   where $0 \log 0:=0$. See Theorem~I.3 and Exercise~I.12 in~\cite{Ho00}. These functions
   are strictly convex on $[-1,1]$ resp.\ $[-2,2].$
     Let $n_k\to\infty$ be a sequence such that $\nu_2(n_k)/n_k\to0.$ 
   Since
      \[
     M_n= \frac{\nu_1(n)}{n} M_n^{(1)} + \frac{\nu_2(n)}{n} M_n^{(2)}
   \]
   and
   \[
      \mathbb{P}\big[|M_{n_k}^{(2)}|\geq 3\big] = 0,
   \]
   we have, for $x>0,$
   \begin{align*}
      \mathbb{P}\big[M_{n_k}\geq x] &= \mathbb{P}[M_{n_k}\geq x,\ |M_{n_k}^{(2)}|<3\big] \\
      &= \mathbb{P}\Big[ \frac{\nu_1({n_k})}{{n_k}} M_{n_k}^{(1)} \geq x - \frac{\nu_2({n_k})}{{n_k}} M_{n_k}^{(2)},
      \ |M_{n_k}^{(2)}|<3\Big] \\
      &\leq \mathbb{P}\Big[ \frac{\nu_1({n_k})}{{n_k}} M_{n_k}^{(1)} \geq x - \frac{3\nu_2({n_k})}{{n_k}} \Big].
   \end{align*}
   Similarly, we deduce the lower bound
   \begin{align*}
     \mathbb{P}[M_{n_k}\geq x] &\geq
     \mathbb{P}\Big[ \frac{\nu_1({n_k})}{{n_k}} M_{n_k}^{(1)} \geq x +\frac{3\nu_2({n_k})}{{n_k}}\Big].
   \end{align*}
   Using the LDP for $M_n^{(1)}$ and~$\nu_1(n_k)/n_k\to1,$ we obtain
   \begin{align*}
     -I^{(1)}(x+\delta) &\leq \liminf_{k\to\infty} \frac{1}{n_k} \log\mathbb{P}[M_{n_k}\geq x]\\
     &\leq \limsup_{k\to\infty} \frac{1}{n_k} \log\mathbb{P}[M_{n_k}\geq x]
     \leq -I^{(1)}(x-\delta) 
   \end{align*}
   for any $\delta>0,$ and by taking $\delta \downarrow0$ we conclude
   \begin{equation}\label{eq:nk}
      \lim_{k\to\infty} \frac{1}{n_k} \log\mathbb{P}[M_{n_k}\geq x] = -I^{(1)}(x),
      \quad x>0.
   \end{equation}
   Analogously, by choosing a sequence $m_k\to\infty$ satisfying $\nu_1(m_k)/m_k\to0,$
   we establish
   \[
      \lim_{k\to\infty} \frac{1}{m_k} \log\mathbb{P}[M_{m_k}\geq x] = -I^{(2)}(x),
      \quad x>0.
   \]
    Suppose now that~$M_n$
   satisfies an LDP with good rate function~$I$ and speed~$s(n).$  
   For $x>0$ and $N\in\mathbb N$, define
   \[
     B_N := (x-1/N, x + 1/N).
   \]
   Then, the assumed LDP implies
   \begin{equation}\label{eq:LDP nk}
     \liminf_{k\to\infty}\frac{1}{s(n_k)} \log \mathbb{P}[M_{n_k} \in B_{N+1}]
     \geq -I(B_{N+1}), \quad N\in\mathbb N.
   \end{equation}
   By~\eqref{eq:nk} and the strict convexity of~$I^{(1)},$ we have
   \begin{equation}\label{eq:I1}
     \log \mathbb{P}[M_{n_k} \in B_{N+1}] = -I^{(1)}(B_{N+1}) n_k\big(1+\mathrm{o}(1)\big),
     \quad k\to\infty.
   \end{equation}
   If $x>1$, then $I^{(1)}(B_{N+1}) =\infty$ for large~$N,$ and~\eqref{eq:LDP nk} and~\eqref{eq:I1} imply
   $I(B_{N+1}) =\infty$ for large~$N.$ By lower semi-continuity, for $N\to\infty$ we get
   \begin{equation}\label{eq:I inf}
     I(x) = \infty, \quad x>1.
   \end{equation}
   For $0<x\leq 1,$ $I^{(1)}(B_{N+1})$ is finite, and~\eqref{eq:LDP nk} and~\eqref{eq:I1} imply
   \[
     I^{(1)}(B_{N+1})\limsup_{k\to\infty} \frac{n_k}{s(n_k)} \leq I(B_{N+1}), \quad N\in\mathbb N.
   \]
   Again, by lower semi-continuity, taking $N\to\infty$ yields
   \begin{equation}\label{eq:I1 I 1}
     I^{(1)}(x)\limsup_{k\to\infty} \frac{n_k}{s(n_k)} \leq I(x),\quad
     0<x\leq1.
   \end{equation}
   Analogously, we can use the upper LDP bound
   \[
     \limsup_{k\to\infty}\frac{1}{s(n_k)} \log \mathbb{P}[M_{n_k} \in \mathrm{cl}(B_{N+1})]
     \leq -I(\mathrm{cl}(B_{N+1}))
     \leq -I(B_N), \quad N\in\mathbb N,
   \]
   to prove
   \begin{equation}\label{eq:I1 I}
     I^{(1)}(x)\liminf_{k\to\infty} \frac{n_k}{s(n_k)} \geq I(x),\quad
     0<x\leq1.
   \end{equation}
   Putting~\eqref{eq:I1 I 1} and~\eqref{eq:I1 I} together yields
   \begin{equation}\label{eq:ell1}
     I(x) = I^{(1)}(x) \ell_1, \quad 0<x\leq1,
   \end{equation}
   where
   \[
     \ell_1 : =\lim_{k\to\infty} \frac{n_k}{s(n_k)}
   \]
   exists in $[0,\infty]$ and is independent of~$x$.
   Repeating the same steps with~$m_k$ instead of~$n_k$ shows
   \begin{equation}\label{eq:rep}
     I(x) = I^{(2)}(x) \ell_2 := I^{(2)}(x) \lim_{k\to\infty} \frac{m_k}{s(m_k)}, \quad 0<x\leq2,
   \end{equation}
   and so $I^{(1)}(x) \ell_1=I^{(2)}(x) \ell_2$ for $0<x\leq 1.$
   {}From the expansions
   \begin{align*}
     I^{(1)}(x) &= \tfrac12 x^2 + \tfrac{1}{12}x^4 + \tfrac{1}{30} x^6 + \mathrm{O}(x^8), \\
     I^{(2)}(x) &= \tfrac18 x^2 + \tfrac{1}{192}x^4 + \tfrac{1}{1920}x^6
          + \mathrm{O}(x^8), \quad x\downarrow 0,
   \end{align*}
   we see that this implies $(\ell_1,\ell_2)=(\infty,\infty)$ or $(\ell_1,\ell_2)=(0,0).$
   The latter is impossible, since~\eqref{eq:I inf} and~\eqref{eq:rep} yield
   \[
     \infty= I(\tfrac32) = I^{(2)}(\tfrac32) \ell_2,
   \]
   which requires $\ell_2=\infty,$ as $I^{(2)}(\tfrac32)$ is finite. To finish the proof,
   we must infer a contradiction from $(\ell_1,\ell_2)=(\infty,\infty).$
   Indeed, \eqref{eq:I inf} and~\eqref{eq:ell1} would then imply $I(x)=\infty$ for
   all $x>0,$ and so
   \[
     \mathbb{P}[M_n \geq 1] = 0,\quad n\in\mathbb N.
   \]
   This is wrong, because $\{M_n \geq 1\}$ contains the event
   \[
     \{ X_k=1\ \text{for}\ k\leq n, k\in K_1 \}\, \cap\, \{ X_k=2\ \text{for}\ k\leq n, k\in K_2 \},
   \]
   which has positive probability.

\bibliographystyle{siam}
\bibliography{../gerhold}

\end{document}